\documentclass[letterpaper, DIV=17]{scrartcl}

\usepackage{amssymb, amsthm, mathtools,bbm}
\usepackage[square,sort,comma,numbers]{natbib}
\usepackage{comment}

\usepackage[none]{hyphenat}
\usepackage[english]{babel}
\usepackage{enumerate}
\usepackage{graphicx}%
\usepackage{subcaption}
\usepackage{dsfont}%
\usepackage{float}%
\usepackage{hyperref}
\usepackage{color}
\usepackage[utf8]{inputenc}
\usepackage{graphicx}
\usepackage{verbatim} 
	\usepackage{hyperref}
	\usepackage{comment}
	\usepackage{xfrac} 
	\usepackage{tensor} 
	\usepackage{listings}  

	\usepackage{pgfplots}              
	\pgfplotsset{compat=newest}
	\usepgfplotslibrary{fillbetween}

	\usepackage[title]{appendix}

	\usepackage[shortcuts]{extdash}  

\hypersetup{pdftitle={On Normality Preserving Operations},
  pdfauthor={Chokri Manai},
  pdfsubject={Normality preserving numbers and maps},
  pdfkeywords={normal numbers, deterministic numbers, normality preserving maps}}
\setkomafont{disposition}{\normalfont\bfseries}
\newtheorem{theorem}{Theorem}[section]
\newtheorem{cor}[theorem]{Corollary}
\newtheorem{prop}[theorem]{Proposition}
\newtheorem{lemma}[theorem]{Lemma}
\theoremstyle{definition}
\newtheorem{defn}[theorem]{Definition}

\numberwithin{equation}{section}
\newcommand{\E}{\mathbb E}
\newcommand{\Prob}{\mathbb P}
\newcommand{\N}{\mathbb N}
\newcommand{\Z}{\mathbb Z}
\newcommand{\Q}{\mathbb Q}
\newcommand{\R}{\mathbb R}
\newcommand{\T}{\mathbb T}
\newcommand{\1}{\mathbf 1}
\newcommand{\calN}{\mathcal N}
\newcommand{\calD}{\mathcal D}
\DeclareMathOperator{\supp}{supp}
\DeclareMathOperator{\dist}{dist}
\DeclareMathOperator{\Lip}{Lip}

\begin{document}
\title{On Normality Preserving Operations}
\author{Chokri Manai\\[2pt]
\small Courant Institute of Mathematical Sciences, New York University, USA; \\\small and Bonn University, Institute for Applied Mathematics, Bonn, Germany}
\date{}
\maketitle

\begin{abstract}
We study arithmetic operations and maps which preserve normality. Revisiting a recent argument of Dayan, Ganguly and Weiss, we give a streamlined,
self-contained presentation of the probabilistic proof that the non-zero
rationals are exactly the multiplicative normality preserving numbers in a
fixed integer base. We also include the extension to all bases simultaneously,
using the theorem of Hochman and Shmerkin. In the additive case, we extend Rauzy's characterization of deterministic numbers in a fixed integer base to absolutely normal numbers. Namely, a translation preserves absolute normality if and only if its parameter is deterministic in every integer base. We prove this extension by a sparse random perturbation argument.
Finally, we show that all locally $C^2$ normality preserving maps need to be affine-linear. The regularity assumption cannot be weakened to $C^{1,1}$: we construct a
non-affine $C^{1,1}$ diffeomorphism which preserves normality in every
integer base, and hence absolute normality.
\end{abstract}

\smallskip
\noindent\textbf{Keywords.} Normal numbers; deterministic numbers; Weyl's criterion;
normality preserving maps; Bernoulli measures; Fourier estimates. \\
\noindent{\textbf{MSC}: Primary 11K16; \, Secondary  37A35, 37A44, 42A38.}

\section{Introduction}

The study of normal numbers goes back to Borel who proved that almost
every real number is absolutely normal~\cite{Borel09}. Since then, there
have been many efforts to obtain a better understanding of their
properties~\cite{Bug12,DT97,Harman98}. Beautiful results concern the
construction and computation of normal numbers
\cite{ABSS17,BF02,BHS13,CE46,Champernowne1933}, the behavior of normality in
different bases~\cite{Cass59,Schmidt1960,Schmidt1962}, its relation to
fractal measures~\cite{ABS22,HS15}, and the interplay of normality and
non-linear maps~\cite{BB25,ManaiAdv, ManaiBernoulli,ManaiTranscendence}.
Yet, normal numbers remain largely mysterious. In particular, establishing
normality for familiar constants such as $\sqrt{2}$, $e$ or $\pi$ remains
a formidable problem; see~\cite{Bug12,BBCP04}. A related question asks
which arithmetic operations and which functions preserve normality.
This is the subject of the present work.

We recall first some standard terminology. Given an integer base $b\ge2$,
we use the canonical $b$-ary expansion of the fractional part of a real
number $x$, choosing the terminating expansion when there are two choices.
The number $x$ is simply $b$-normal if all digits are asymptotically
equidistributed. If every finite word $w$ occurs with asymptotic frequency
$b^{-|w|}$, then $x$ is $b$-normal. We denote the latter set by $\calN_b$.
Finally, $x$ is absolutely normal if it is normal in every integer base,
and we write
\[
 \calN:=\bigcap_{b\ge2}\calN_b.
\]
Throughout the paper, a map preserves one of these classes if it sends
every member of the class in its domain to a member of the same class.
Thus, preservation for maps means forward preservation; it does not
require the inverse map to preserve normality.

A very natural question is which numbers $\gamma$ preserve normality in
a multiplicative sense. Wall proved that all non-zero rational numbers
have this property~\cite{Wall49}. The converse is a consequence of the
work of Dayan, Ganguly and Weiss~\cite{DGW24}. Their Theorem~4 applies to
irrational affine images of the Bernoulli measures used below. Their
Corollary~5, concerning irrational dilations of the middle-thirds Cantor
set, also gives the converse for absolutely normal numbers. We record
the resulting characterization in the following form.

\begin{theorem}[Wall \cite{Wall49} and Dayan--Ganguly--Weiss \cite{DGW24}]
\label{thm:1}
Let $b\ge2$ be an integer. Then
\begin{equation}\label{eq:bpreserve}
 \{\gamma\in\R:\ \gamma x\in\calN_b
                 \text{ for all }x\in\calN_b\}
   =\Q\setminus\{0\}.
\end{equation}
Similarly,
\begin{equation}\label{eq:allpreserve}
 \{\gamma\in\R:\ \gamma x\in\calN
                 \text{ for all }x\in\calN\}
   =\Q\setminus\{0\}.
\end{equation}
\end{theorem}

We revisit the proof of Dayan, Ganguly and Weiss in the one-dimensional
Bernoulli setting to give a streamlined and self-contained presentation
of its probabilistic argument. The proof for a fixed base uses a Markov
chain, uniqueness of its stationary measure, and a martingale argument.
The extension to independent bases uses the theorem of Hochman and
Shmerkin~\cite{HS15}, as in~\cite{DGW24}. The same argument
shows that for every irrational $\gamma$ both exceptional sets
\[
 E_{\gamma,b}:=\{x\in\calN_b:\ \gamma x\notin\calN_b\},
 \qquad
 E_\gamma:=\{x\in\calN:\ \gamma x\notin\calN\}
\]
have Hausdorff dimension $1$ and by Borel's theorem they are null sets. All proofs can be found in Section~\ref{sec:multiplication}.

The additive analogue to Theorem~\ref{thm:1} is quite different. Let us
define
\begin{align}
 \calD_b&:=\{a\in\R:\ x+a\in\calN_b
                       \text{ for all }x\in\calN_b\},\label{eq:Db}\\
 \calD&:=\{a\in\R:\ x+a\in\calN
                       \text{ for all }x\in\calN\}.\label{eq:D}
\end{align}
Rauzy's seminal contribution identifies $\calD_b$ as the set of numbers
which are completely deterministic in base $b$~\cite{Rauzy76}.
Roughly speaking, all empirical limit measures of their digit sequence
have entropy zero. We give the precise definition in
Section~\ref{sec:determinism}; see also~\cite{BD25} for a modern ergodic
treatment of Rauzy's theorem and related arithmetic questions.

\begin{theorem}\label{thm:addition}
Let $b\ge2$ and $a\in\R$. Then $a\in\calD_b$ if and only if $a$ is
completely deterministic in base $b$. Moreover,
\begin{equation}\label{eq:absoluteD}
 \calD=\bigcap_{b\ge2}\calD_b.
\end{equation}
Thus, a translation preserves absolute normality if and only if its
parameter is completely deterministic in every integer base. In both
statements, forward preservation is equivalent to preservation in both
directions:
\[
 x\in\calN_b\ \Longleftrightarrow\ x+a\in\calN_b,
 \qquad\text{respectively}\qquad
 x\in\calN\ \Longleftrightarrow\ x+a\in\calN.
\]
\end{theorem}

The fixed-base statement is Rauzy's theorem \cite{Rauzy76, BD25}. The novel part concerns~\eqref{eq:absoluteD}. One inclusion is immediate, but
the other does not follow merely by intersecting the fixed-base
statements. A counterexample in one base need not be absolutely normal.
We overcome this difficulty by adding the same sparse random number to
a normal and a nonnormal number. The perturbation makes the first number
absolutely normal almost surely, while every realization preserves the
failure of normality of the second number in the original base. The proof is given in Section~\ref{sec:determinism} after a short review on deterministic numbers and Rauzy's original result.

We next turn to normality preserving functions. It turns out that a local $C^2$ assumption excludes all non-affine functions as possible candidates.  The deterministic classes
in Theorem~\ref{thm:addition} give the complete characterization of their translation parameter.

\begin{theorem}\label{thm:2}
Let $I\subset\R$ be a nonempty open interval and let
$f\in C^2_{\mathrm{loc}}(I)$. Fix an integer $b\ge2$.
Then
\[
 f(\calN_b\cap I)\subseteq\calN_b
\]
if and only if either $f\equiv c$ with $c\in\calN_b$, or
\[
 f(x)=ax+c,\qquad a\in\Q\setminus\{0\},\quad c\in\calD_b.
\]
Similarly,
\[
 f(\calN\cap I)\subseteq\calN
\]
if and only if either $f\equiv c$ with $c\in\calN$, or
\[
 f(x)=ax+c,\qquad a\in\Q\setminus\{0\},\quad
 c\in\calD=\bigcap_{b\ge2}\calD_b.
\]
\end{theorem}

The proof uses the nonvanishing-second-derivative case of the
pushforward theorem of Baker and Banaji~\cite{BB25}. If $f$ is not
affine-linear, we restrict to an interval on which its first and second
derivatives do not vanish. Its local inverse then sends almost every
point of a suitable nonnormal Bernoulli measure to an absolutely normal
number. This contradicts preservation. The affine case follows from the
probabilistic statement used in Theorem~\ref{thm:1} and the additive
characterization. The proof is carried out in Section~\ref{sec:C2}.

Very interestingly, the conclusion changes as soon as the $C^2$-assumption is slightly weakened. Indeed, Theorem~\ref{thm:2} fails for $C^{1,1}$-functions, i.e. continuously
differentiable functions with a Lipschitz derivative.

\begin{theorem}\label{thm:C11}
There is a non-affine increasing $C^{1,1}$ diffeomorphism
$f:\R\to\R$ such that
\begin{equation}\label{eq:C11preserve}
 f(\calN_b)\subseteq\calN_b\quad\text{for every integer }b\ge2.
\end{equation}
In particular, the same map preserves absolute normality. Both $f$ and
$f^{-1}$ have globally Lipschitz derivatives.
\end{theorem}

Our construction uses a compact set of absolutely normal numbers of
positive measure. We first construct a non-affine diffeomorphism $g$
which is rational-affine on every complementary interval of this set.
Such a map preserves nonnormality in every base. Its inverse is the
required normality preserving map. The main point is to arrange two
rational moments at each complementary interval while keeping the
second derivative bounded. This is done by small, localized
corrections. The construction can be found in Section~\ref{sec:C11}.

\section{Multiplicative operations}\label{sec:multiplication}

We fix some notation. We write $\T:=\R/\Z$ for the unit torus and
$\|u\|:=\dist(u,\Z)$ for the corresponding lattice distance. We abbreviate
the Fourier modes by $e(t):=\exp(2\pi i t)$. For a probability measure
$\mu$, let
\[
 \widehat\mu(t):=\int e(ty)\,d\mu(y).
\]
Wall's theorem will be used in its usual rational-affine form:
\begin{equation}\label{eq:Wall}
 x\in\calN_b\quad\Longleftrightarrow\quad rx+s\in\calN_b
 \qquad(r\in\Q\setminus\{0\},\ s\in\Q).
\end{equation}
The same equivalence holds for $\calN$. We also use the equivalence
$\calN_b=\calN_{b^r}$ for positive integers $r$; see~\cite{Wall49,Bug12}.
By Weyl's criterion~\cite{Weyl1916}, $x\in\calN_b$ if and only if
\begin{equation}\label{eq:Weyl}
 A_N^{(b)}(h;x):=\frac1N\sum_{n=0}^{N-1}e(hb^nx)\longrightarrow0
 \quad\text{for every }h\in\Z\setminus\{0\}.
\end{equation}

\subsection{Proof of Theorem~\ref{thm:1}: the \texorpdfstring{$b$}{b}-normal case}
We introduce the probabilistic model which we use to derive our main results. We fix throughout this section an integer base $b \geq 2$. Let 
\begin{equation}\label{eq:p}
\mathbf p=(p_0,\dots,p_{b-1}),
\qquad p_d>0,\qquad \sum_{d=0}^{b-1}p_d=1,
\end{equation}
be a probability vector. We further suppose that $\mathbf p \neq \left(\frac1b, \ldots, \frac1b \right)$ is not the constant vector.
Given such a vector $\mathbf p$, let $D_1, D_2, \ldots$ be i.i.d. random variables with law
\begin{equation}\label{eq:Digit}
\Prob(D_j=d)=p_d.
\end{equation}
Our probabilistic model concerns the random number
\begin{equation}\label{eq:Y}
Y :=\sum_{j=1}^{\infty}D_j b^{-j}
\end{equation}
and since $\mathbf p$ is not uniform, $Y$ is almost surely not $b$-normal. The law $\mu_{b,\mathbf p}$ of $Y$ is a probability measure on $[0,1]$ with Fourier transform
\begin{equation}\label{eq:fourier-product}
\widehat\mu_{b,\mathbf p}(t)
=\int e(t y)\,d\mu_{b,\mathbf p}(y)
=\prod_{j=1}^{\infty}\Phi(t b^{-j}),
\qquad
\Phi(u) :=\sum_{d=0}^{b-1}p_d e(d u).
\end{equation}
The infinite product identity follows from the independence of the digits.
We further remark that 
\begin{equation}\label{eq:phi-one}
|\Phi(u)|=1
\quad\Longleftrightarrow\quad
u\in\mathbb Z 
\end{equation}
since $p_d>0$ for every $d = 0, 1, \ldots, b-1.$
The key probabilistic statement is the following special case of \cite[Theorem~4]{DGW24}. We revisit its proof in this Bernoulli setting.

\begin{prop}\label{prop:irrational}
Let $b\ge 2$, let $\mathbf p$ be as above, and let $\gamma\notin\Q$. Then, for every $\eta\in\R$,
\begin{equation}\label{eq:key}
\mathbb{P} \left(\gamma Y + \eta \text{ is normal to base }b \right) = 1.
\end{equation}
In other words, $\gamma Y + \eta$ is $b$-normal for $\mu_{b,\mathbf p}$-almost every $Y$.
\end{prop}

For comparison with \cite{DGW24}, the law of $\gamma Y+\eta$ is the Bernoulli measure for the maps
    \[
      x\longmapsto b^{-1}x+b^{-1}\gamma d+(1-b^{-1})\eta,
      \qquad d=0,\ldots,b-1.
    \]
    Two translation parts differ by $\gamma/b$, which is irrational.
    
    Let us first quickly show how \eqref{eq:bpreserve} follows from Proposition~\ref{prop:irrational}.
\begin{proof}[Proof of Theorem~\ref{thm:1}: the $b$-normal case]
Let us write 
$$ \mathcal{P}_b := \{\gamma \in \R \, | \, \gamma x \in \mathcal{N}_b \text{ for all } x \in \mathcal{N}_b \}.$$
The inclusion $\Q\setminus\{0\} \subset \mathcal{P}_b$ follows from  Wall's theorem. It is clear that $0 \notin \mathcal{P}_b$. Let $\gamma \in \R \setminus \Q$ and it remains to show that $\gamma \notin \mathcal{P}_b$. 

 Choose any non-uniform
probability vector $\mathbf p$ as in \eqref{eq:p} and let $Y$ be the random number from \eqref{eq:Y} with  law
$\mu_{b,\mathbf p}$. We apply Proposition~\ref{prop:irrational} with $\gamma^{-1}$ and $\eta = 0$. Then $Y$ is not $b$-normal for
$\mu_{b,\mathbf p}$-almost every $Y$, but  
$\gamma^{-1} Y$ is $b$-normal for $\mu_{b,\mathbf p}$-almost every $Y$.
In particular, there exists a number $y$ with both properties. But then $\gamma^{-1} y$ is $b$-normal and $y$ is not, so $\gamma \notin \mathcal{P}_b$.
\end{proof}

At first glance, Proposition~\ref{prop:irrational} seems to be obviously wrong. Fixing a vector $\mathbf p$, one is tempted to think that a Liouville-type number $\gamma$ leads to a counterexample as its nonzero digits are so rare that it can never ensure enough mixing. Let us consider the special case $b = 2$ and let $\gamma = \sum_{k = 1}^{\infty} 2^{-r_k}$, where we think of $r_k$ being a very fast growing sequence, say $k!$. Then,
$$ \gamma Y = \sum_{n = 1}^{\infty} \sum_{k, \, r_k < n} D_{n-r_k} 2^{-n}.$$
Let us first pretend that the digit at position $n$ is solely determined by the $n$-level term $\sum_{k, \, r_k < n} D_{n-r_k} 2^{-n}$. Then, the distribution of the digit $n$ would coincide with the parity of the binomial random variable $\sum_{k, \, r_k < n} D_{n-r_k}$. Even if $r_k$ is growing fast, ultimately there will be many independent terms and this probability converges to $\frac12$. Of course, this picture is too naive as there is a further carry interaction between the different terms. The heuristics is as follows. If $\gamma$ is rational and, thus, its digits are eventually periodic, there is a high resonance between the level sets resembling a periodic crystal. This destroys the above parity reasoning. On the other hand, if $\gamma$ is irrational, these resonances are not present due to the aperiodicity. So even Liouville-type numbers are not valid counterexamples to Proposition~\ref{prop:irrational}. A direct digit-by-digit rigorous justification of the above sketched heuristics is not feasible and we employ instead Weyl's criterion. But the above ideas are still visible in our proof which occupies the rest of this section.

We start with the following simple observation which encodes the absence of resonances for irrational numbers.
\begin{lemma}\label{lem:product-zero}
If $\gamma\notin\Q$, then
\begin{equation}
\lim_{M \to \infty} \prod_{r=0}^{M-1}|\Phi(\gamma b^r)| = 0.
\end{equation}
\end{lemma}

\begin{proof}
Note that the factors have modulus at most $1$. Suppose the product does not tend to $0$. This is only possible if 
$|\Phi(\gamma b^r)|\to 1$. By \eqref{eq:phi-one}, continuity of $\Phi$ and compactness modulo one, this implies
$$ \|\gamma b^r\|\longrightarrow 0,$$
where $\|\cdot\|$ is the lattice distance. In particular, there is a natural number $r_0$ such that for all $r \geq r_0$  we have $\|\gamma b^r\|<1/(2b)$. However if $\|x \| < \frac{1}{2b}$, then $\|b x \| = b \,\|x\|$.  Iterating this identity gives
 $\|\gamma b^{r_0+s}\|=b^s\|\gamma b^{r_0}\|$ for every $s\geq 0$. This is only possible if $\|\gamma b^{r_0}\|=0$. This forces $\gamma\in b^{-r_0}\mathbb Z\subset\Q$, and this is a contradiction.
\end{proof}

The main idea for the proof of Proposition~\ref{prop:irrational} is a Markov chain argument. We want to identify the equidistribution of the digits as unique stationary measure of the digit process corresponding to $\gamma Y$.
Consider the Markov chain on $\T$ defined by
\begin{equation}\label{eq:chain}
X_0=x_0, \qquad X_{n+1}=bX_n+\gamma D_{n+1}\pmod 1.	
\end{equation}
The Markov chain is uniquely characterized by its transition operator $P: C(\T) \to C(\T)$,
\begin{equation}\label{eq:transition}
(Pf)(x) :=\sum_{d=0}^{b-1}p_d f(bx+\gamma d).
\end{equation}
 acting on the space of continuous functions $C(\T)$. Recall that a stationary measure $\mu$ for a Markov chain is characterized by the identity
 $$ \int (Pf)(x) d \mu(x) = \int f(x) d \mu(x) $$
 for all bounded continuous functions $f$. 
 
    We shall use the strong law for bounded martingale differences. To keep the probabilistic argument self-contained, let us recall its short proof. Suppose $(M_n)$ is a sequence of martingale differences with $|M_n|\leq C$, and write $S_N=\sum_{n=1}^NM_n$. For $m<n$, conditional expectation gives $\E(M_n\overline{M_m})=0$. Thus,
    \[
      \E|S_N|^2=\sum_{n=1}^N\E|M_n|^2\leq C^2N.
    \]
    Given $\varepsilon>0$, Chebyshev's inequality shows that
    \[
      \sum_{j=1}^{\infty}\Prob\bigl(|S_{j^2}|>\varepsilon j^2\bigr)
      \leq \frac{C^2}{\varepsilon^2}\sum_{j=1}^{\infty}\frac1{j^2}<\infty.
    \]
    By Borel--Cantelli, and then a countable intersection over $\varepsilon=1,1/2,1/3,\ldots$, we obtain $S_{j^2}/j^2\to0$ almost surely. To pass from squares to all $N$, choose $j$ with $j^2\leq N<(j+1)^2$. Then,
    \begin{equation}\label{eq:martingalegaps}
      \left|\frac{S_N}{N}-\frac{S_{j^2}}{j^2}\right|
      \leq 2C\frac{N-j^2}{N}
      \leq 2C\frac{2j+1}{j^2}\longrightarrow0.
    \end{equation}
    This proves the asserted strong law; see also \cite{Chow67}.

    \begin{lemma}\label{lem:stationary}
 	If $\gamma\notin\Q$, then the uniform measure $m$ on $\T$ is the unique stationary measure for the chain \eqref{eq:chain}. Moreover, for every starting point
 	$x_0\in\T$ and for $\mu_{b,\mathbf p}$-almost every digit sequence $(D_j)$,
 	\begin{equation}\label{eq:equilibrium}
 	\frac1N\sum_{n=0}^{N-1} f(X_n)\longrightarrow \int_{\T}f\,dm
 	\qquad(f\in C(\T)).
 	\end{equation}
 \end{lemma}
 The hidden event in the $\mu_{b,\mathbf p}$-almost every statement may depend on the starting point $x_0$. We will only use a fixed starting point in the proof of Proposition~\ref{prop:irrational}, so this issue is harmless for our purposes.
 \begin{proof}
 	The uniform measure is stationary because the maps $x\mapsto bx+\gamma d$ preserve the uniform measure on $\T$ for each fixed $d$.
 	Indeed, by the translation invariance of the Lebesgue measure it suffices to show the claim for the map $x\mapsto bx$. But this follows easily since $b$ is an integer.
 	
 	Let now $\nu$ be any stationary probability measure with Fourier coefficients $\widehat\nu(k)=\int e(kx)\,d\nu(x)$.  Stationarity gives, for $k\in\mathbb Z$,
 	$$ \widehat\nu(k)= \mathbb E_{D_1} \int e(k(by + \gamma D_1))\,d\nu(y) = \Phi(k\gamma)\widehat\nu(kb), $$
 	where all expressions are understood to be $\pmod 1$.
 	Iterating this identity $M$ times, gives
 	$$ \widehat\nu(k) =\widehat\nu(kb^M)\prod_{r=0}^{M-1}\Phi(k\gamma b^r).$$
 	If $k\neq0$, then $k\gamma$ is irrational, so by Lemma~\ref{lem:product-zero} the product tends to $0$. Since $|\widehat\nu(kb^M)|\leq 1$, we get
 	$\widehat\nu(k)=0$ for every $k\neq0$. The Fourier coefficients uniquely characterize the underlying measure. Thus,  $\nu=m$.
 	
 	It remains to prove the pointwise assertion. Fix an $x_0 \in \T$ and suppose $X_0 = x_0$.  For a sample path define the empirical measures
 	$$  L_N=\frac1N\sum_{n=0}^{N-1}\delta_{X_n},$$
 	These are random measures as they depend on the random variables $D_1, \ldots, D_{N-1}$. Given a trial function $f\in C(\T)$, we consider the differences
 	$$ M_{n+1}=f(X_{n+1})-Pf(X_n).$$
 	By construction, these are bounded martingale differences with respect to the natural filtration $\mathcal F_n := \sigma(D_1,\ldots, D_n)$.
 	By the strong law for bounded martingale differences recalled above (see also \cite{Chow67}), it follows
 	$$ \frac1N\sum_{n=0}^{N-1}M_{n+1}\longrightarrow0
 	\qquad\text{almost surely}. $$
 	Also
 	\[ \frac1N\sum_{n=0}^{N-1}f(X_{n+1}) -\frac1N\sum_{n=0}^{N-1}f(X_n)
 	 =\frac{f(X_N)-f(X_0)}{N}\longrightarrow0.\]
 	Therefore almost surely every weak subsequential limit $L$ of $(L_N)$  satisfies
 	$$ \int f\,dL=\int Pf\,dL$$
 	for every $f$ in a fixed countable dense subset of $C(\T)$, and hence for every $f\in C(\T)$. Thus $L$ is stationary. Since the stationary measure is unique, $L=m$ for every weak subsequential limit $L$. By Prokhorov's theorem, the space of probability measures on $\T$ is compact with respect to the weak convergence. Hence, each subsequence $L_{N_k}$ has a further weakly converging subsubsequence $L_{N_{k_l}}$ with limit $m$. This implies that the full sequence converges $L_N \xrightarrow{w} m$ almost surely. This is by definition equivalent to \eqref{eq:equilibrium}.
 \end{proof}
 
 We are now ready to spell out the proof of Proposition~\ref{prop:irrational}.
 
 \begin{proof}[Proof of Proposition~\ref{prop:irrational}]
 	We argue via the well-known Weyl criterion \cite{Weyl1916, Bug12}. Fix $h\in\mathbb Z\setminus\{0\}$ and note that then $h \gamma $ is irrational.  For $n\ge0$, let us introduce the two random sequences
 	$$ A_n :=\sum_{j=1}^{n}D_j b^{n-j}, \qquad R_n :=\sum_{j=1}^{\infty}D_{n+j}b^{-j}, $$
 	and we agree that $A_0 := 0$. Thus,
 	$$ b^nY= A_n+R_n, $$
 	where $A_n$ encodes the integer part of $b^n Y$ and $R_n$ the fractional part. Let us set 
 	$$X_n := h(\gamma A_n+b^n\eta)\pmod 1 $$ 
 	and observe that for $n \geq 0$
 	$$ X_{n+1}= h(\gamma A_{n+1}+b^{n+1}\eta) =  h (\gamma b A_n + \gamma D_{n+1} + b^{n+1} \eta) = b X_n + \gamma h D_{n+1}  \pmod1. $$
 	Hence, $X_n$ is exactly the Markov chain from \eqref{eq:chain} with $h \gamma$ as parameter started at $X_0= h\eta$.
 	
 	We start with a little upgrade of Lemma~\ref{lem:stationary}. We want to consider trial functions which may depend on a finite part of the future.  Let $L\ge1$ and $g:\{0,\dots,b-1\}^L\to\mathbb C$ and consider the bounded process $G_n :=g(D_{n+1},\dots,D_{n+L})$
 	depending on the next $L$ digits. 
 	We claim
 	\begin{equation}\label{eq:finite-future}
 		\frac1N\sum_{n=0}^{N-1} e(X_n)G_n
 		\longrightarrow 0
 		\qquad\text{almost surely}.
 	\end{equation}
 	Let us center $G_n$ as $G_n = (G_n -\bar g) + \bar g$ with  $\bar g :=\E G_n$. We analyze both contributions separately.  By Lemma~\ref{lem:stationary},
 	$$ \frac1N\sum_{n=0}^{N-1}e(X_n)\bar g\longrightarrow \bar g\int_{\T} e(t )\,dm(t) = 0.$$
 	For the centered part, we split the indices $n$ into residue classes modulo $L+1$. That is 
 	$$ \frac1N\sum_{n=0}^{N-1} e(X_n)(G_n - \bar g) = \sum_{k = 0}^{L} \frac1N\sum_{n < N, \, n\bmod (L+1) = k} e(X_n)(G_n - \bar g)$$
 	and it is enough to prove the convergence for each residue class $n_j = k + j(L+1)$. The trick is that along each residue class $k$, the random variables $	e(X_{n_j})(G_{n_j}-\bar g)$ are again bounded martingale differences with respect to the filtration $\mathcal G_j := \sigma(D_1,\ldots, D_{n_j + L})$. We may omit the first term of the residue class. For $j\geq1$, note that $n_{j-1}+L=n_j-1$ and hence $\sigma(\mathcal G_{j-1},D_{n_j})=\mathcal F_{n_j}$. Indeed $e(X_{n_j})(G_{n_j}-\bar g)$ is $\mathcal{G}_j$-measurable, and 
 	\begin{align*} \mathbb{E}[e(X_{n_j})(G_{n_j}-\bar g) \, | \, \mathcal{G}_{j-1} ] & = \mathbb{E}[ \mathbb{E}[e(X_{n_j})(G_{n_j}-\bar g) \, | \, \sigma(\mathcal{G}_{j-1}, D_{n_j})] \, | \, \mathcal{G}_{j-1} ] \\
 		&= \mathbb{E}[e(X_{n_j})\mathbb{E}[(G_{n_j}-\bar g) \, | \, \sigma(\mathcal{G}_{j-1}, D_{n_j})] \, | \, \mathcal{G}_{j-1} ] = 0.
 	\end{align*}
    Here, we used the tower property, that $X_{n_j}$ is $\sigma(\mathcal{G}_{j-1}, D_{n_j})$-measurable, while $G_{n_j}-\bar g$ is centered and independent from $\sigma(\mathcal{G}_{j-1}, D_{n_j})$.  Applying again the strong law for bounded martingale
 	differences,  gives convergence to $0$ along each residue class, and
 	hence \eqref{eq:finite-future}.
 	
 	The rest of the proof is simple: due to \eqref{eq:finite-future} we can control an arbitrarily large block $L \geq 1$ of the remainder $R_n$ and the remaining tail becomes negligible. To be precise,  approximate the remainder $R_n$ by
 	$$R_n^{(L)}=\sum_{j=1}^{L}D_{n+j}b^{-j}.$$
 	Then, $|R_n-R_n^{(L)}|\le b^{-L}$ holds uniformly in $n$.
 	The Lipschitz continuity of the Fourier modes implies
 	$$ \left|e(h \gamma R_n)-e(h \gamma R_n^{(L)})\right|
 	\leq 2\pi |h\gamma|b^{-L}.$$
 	Taking $G_n :=e(h\gamma R_n^{(L)})$
 	in \eqref{eq:finite-future}, we obtain 
 	\begin{align*}
 	\limsup_{N \to \infty} \left|\frac1N\sum_{n=0}^{N-1}e(h b^n(\gamma Y + \eta)) \right|
 	&= \limsup_{N \to \infty} \left|\frac1N\sum_{n=0}^{N-1}e( X_n)e(h \gamma R_n) \right| \\
 	&\leq \limsup_{N \to \infty} \left|\frac1N\sum_{n=0}^{N-1}e(X_n) G_n \right| + 2\pi |h\gamma|b^{-L} = 2\pi |h\gamma|b^{-L}
 	\end{align*}
 	almost surely. We first intersect the events of probability one over the countably many $L$. Since $L$ can then be chosen arbitrarily large, we conclude that $\frac1N\sum_{n=0}^{N-1}e(h b^n(\gamma Y + \eta)) \to 0$ almost surely. Intersecting also over the non-zero integers $h$, Weyl's criterion implies that $\gamma Y+\eta$ is $b$-normal for $\mu_{b,\mathbf p}$-almost every $Y$.
 \end{proof}
 
\subsection{Proof of Theorem~\ref{thm:1}: the absolutely normal case}

The extension to absolutely normal numbers turns out to be easy as we can rely on a beautiful theorem by Hochman and Shmerkin \cite{HS15}. Recall that two natural numbers $m,n$ are multiplicatively independent if $\frac{\log m}{\log n} \notin \Q$.

\begin{theorem}[Theorem 1.10 in \cite{HS15}]\label{thm:HS}
Let $m,n\ge 2$ be multiplicatively independent. Let $\mu$ be an
ergodic $T_m$-invariant probability measure on $[0,1)$ with positive
entropy, where $T_m x=mx \pmod 1$. Let $f$ be a $C^2$ diffeomorphism on $\R$. Then, $f(x)$ is normal to base
$n$ for $\mu$-almost every $x$.
\end{theorem}

    Let us restrict our probabilistic model to the special case
\begin{equation}\label{eq:Y2}
Y :=\sum_{j=1}^{\infty}D_j 2^{-j},
\end{equation}
where the digits $D_j$ are i.i.d. Bernoulli random variables with success probability $p \in (0,1) \setminus\{1/2\}$. The following sharpens Proposition~\ref{prop:irrational}.
\begin{prop}\label{prop:irrational2}
Let $Y$ be as above and $\gamma\notin\Q$. Then, for every $\eta\in\R$,
\begin{equation}\label{eq:key2}
\mathbb{P} \left(\gamma Y + \eta \text{ is absolutely normal } \right) = 1.
\end{equation}
\end{prop}

Taking Proposition~\ref{prop:irrational2} for granted, the rest of Theorem~\ref{thm:1} follows exactly as the first part. Thus, it suffices to prove Proposition~\ref{prop:irrational2}.
\begin{proof}[Proof of Proposition~\ref{prop:irrational2}]
Since the set of bases is countable, it is enough to prove the almost sure $b$-normality of $\gamma Y + \eta$ for each fixed $b \geq 2$. By Proposition~\ref{prop:irrational}, $\gamma Y + \eta$ is almost surely $2$-normal. This implies almost surely normality for all bases $b = 2^n$ and $n \geq 1$. Let $b \notin \{2^n\}_{n \geq 1}$. Then, $b$ is multiplicatively independent from $2$. Note that the law $\mu$ of $Y$ is $T_2$-invariant and has positive entropy since $p \neq 0,1$. Recall further that the Kolmogorov $0-1$ law implies the ergodicity of $\mu$: an event invariant under the digit shift belongs to the tail $\sigma$-algebra. Its entropy is $-p\log p-(1-p)\log(1-p)>0$. Since $x \mapsto \gamma x + \eta$ is obviously a diffeomorphism on $\R$, we deduce that $\gamma Y + \eta$ is almost surely $b$-normal. This completes the proof.
\end{proof}

\begin{proof}[Proof of Theorem~\ref{thm:1}: the absolutely normal case]
The rational inclusion again follows from Wall's theorem. If
$\gamma\notin\Q$, Proposition~\ref{prop:irrational2} shows that
$\gamma^{-1}Y$ is absolutely normal almost surely, whereas $Y$ is
almost surely not normal in base $2$. This excludes $\gamma$.
\end{proof}

Let us also make the direct connection with~\cite[Corollary~5]{DGW24}
explicit. Let $C$ be the middle-thirds Cantor set and take a Bernoulli
measure on it. That corollary implies that $\gamma^{-1}y$ is absolutely
normal for almost every $y\in C$ when $\gamma$ is irrational. Since
no point of $C$ is $3$-normal, this already proves the irrational
exclusion in~\eqref{eq:allpreserve}.

\begin{cor}\label{cor:dimension}
For every irrational $\gamma$, the sets $E_{\gamma,b}$ and $E_\gamma$
defined in the introduction have Hausdorff dimension $1$ and Lebesgue
measure zero.
\end{cor}

\begin{proof}
Put $p_*:=\max_dp_d$ and $s:=-\log p_*/\log b$. A level-$n$ base-$b$
cylinder has $\mu_{b,\mathbf p}$-mass at most $p_*^n$. An interval $J$
with $b^{-n}\le |J|<b^{-n+1}$ intersects at most $b+2$ such cylinders.
Thus
\[
 \mu_{b,\mathbf p}(J)\le(b+2)|J|^s.
\]
The same estimate, with a different constant, holds after a fixed
non-zero affine scaling. Every full-measure set for such a measure has
Hausdorff dimension at least $s$: apply the displayed estimate to any
interval cover. The proof of~\eqref{eq:bpreserve} gives a full-measure
subset of $E_{\gamma,b}$ for the law of $\gamma^{-1}Y$. Letting a
nonuniform $\mathbf p$ tend to the uniform vector gives dimension $1$.
For $E_\gamma$, use~\eqref{eq:Y2} and let $p\to1/2$. Finally, each
exceptional set is contained in the inverse image of a Lebesgue null
nonnormal set under a non-zero linear map, so is null.
\end{proof}

\section{Deterministic numbers and additive operations}
\label{sec:determinism}

The additive analogue of Theorem~\ref{thm:1} is quite different. For a
fixed base, Rauzy's theorem identifies the admissible translations with
the completely deterministic numbers. We first explain this condition
and its formulation in terms of digit frequencies. We then turn to the
absolutely normal case and prove that a translation preserves absolute
normality precisely when it preserves normality in every base
separately.

\subsection{Deterministic numbers and Rauzy's theorem}

Let us fix an integer base $b\geq2$ and a real number $a$. Write
\[
 \{a\}=0.a_1a_2a_3\cdots
\]
for its canonical base-$b$ expansion. As before, we choose the
terminating expansion whenever there are two choices. The frequencies
of finite words in this sequence need not converge. We therefore keep
track of all possible limiting digit distributions, rather than assume
that there is a single one.

To make this precise, let
\[
 \Omega_b:=\{0,\ldots,b-1\}^{\N},\qquad
 \omega_b(a):=(a_1,a_2,\ldots),
\]
where $\Omega_b$ carries the product topology, and let $\sigma$ denote
the left shift. We consider the empirical measures
\[
 L_{b,N}(a):=\frac1N\sum_{n=0}^{N-1}
                  \delta_{\sigma^n\omega_b(a)}.
\]
For a word $w=(w_1,\ldots,w_k)$, its cylinder is
\[
 [w]:=\{\omega\in\Omega_b:\ \omega_1=w_1,\ldots,\omega_k=w_k\}.
\]
Thus $L_{b,N}(a)([w])$ is exactly the frequency with which $w$ starts
at one of the first $N$ positions of the digit sequence. Blocks are
allowed to overlap, and the last $k-1$ digits of a block may lie beyond
the $N$th position.

Let $V_b(a)$ be the set of weak subsequential limits of $L_{b,N}(a)$
as $N\to\infty$. Since $\Omega_b$ is compact, the space of probability
measures on it is compact as well. In particular, $V_b(a)$ is a
nonempty compact set. Every $\mu\in V_b(a)$ is shift-invariant.
Indeed, for a continuous trial function $F$,
\[
 \int F\circ\sigma\,dL_{b,N}(a)-\int F\,dL_{b,N}(a)
 =\frac{F(\sigma^N\omega_b(a))-F(\omega_b(a))}{N}
 \longrightarrow0.
\]
Passing to a subsequence along which $L_{b,N}(a)$ converges proves
$\int F\circ\sigma\,d\mu=\int F\,d\mu$.

For a probability measure $\mu$ on $\Omega_b$, define its length-$k$
block entropy by
\[
 H_k(\mu):=-\sum_{|w|=k}\mu([w])\log_b\mu([w]),
 \qquad 0\log_b0:=0.
\]
If $\mu$ is shift-invariant, the entropy inequality for two consecutive
blocks gives
\[
 H_{k+\ell}(\mu)\leq H_k(\mu)+H_\ell(\mu).
\]
Here invariance ensures that the second block has the same distribution
as a block of length $\ell$ starting at the first digit. Subadditivity
therefore guarantees the existence of the limit
\[
 h_b(\mu):=\lim_{k\to\infty}\frac{H_k(\mu)}k
          =\inf_{k\geq1}\frac{H_k(\mu)}k.
\]
This is the entropy of the shift, normalized by $\log b$. In
particular, $0\leq h_b(\mu)\leq1$. The uniform Bernoulli measure has
entropy $1$, whereas a measure supported on a finite periodic orbit
has entropy zero.

\begin{defn}\label{def:determinism}
A real number $a$ is completely deterministic in base $b$ if
$h_b(\mu)=0$ for every $\mu\in V_b(a)$.
\end{defn}

A digit sequence can have different
limiting distributions along different subsequences and it is important to note that complete
determinism requires zero entropy for all of them.

There is also a formulation which only involves the digits of $a$.
For a word $w$ of length $k$, put
\[
 p_{b,k,N}^a(w):=\frac1N\#\{0\leq n<N:
                     a_{n+1}\cdots a_{n+k}=w\},
\]
and set
\begin{equation}\label{eq:blockentropy}
 H_{b,k,N}(a):=-\sum_{|w|=k}
                    p_{b,k,N}^a(w)\log_b p_{b,k,N}^a(w).
\end{equation}
The next observation explains exactly how these empirical entropies
recover the preceding definition.

\begin{prop}\label{prop:entropy}
For every $a\in\R$ and $b\geq2$,
\begin{equation}\label{eq:entropyidentity}
 \inf_{k\geq1}\frac1k\limsup_{N\to\infty}H_{b,k,N}(a)
 =\lim_{k\to\infty}\frac1k\limsup_{N\to\infty}H_{b,k,N}(a)
 =\max_{\mu\in V_b(a)}h_b(\mu).
\end{equation}
In particular, complete determinism is equivalent to the vanishing of
these quantities.
\end{prop}

\begin{proof}
Let us write $K=V_b(a)$. For fixed $k$, the function $\mu\mapsto H_k(\mu)$
is continuous: the cylinder sets are both closed and open, there are
only finitely many words of length $k$, and $-t\log_b t$ is continuous
on $[0,1]$ with the convention at $0$ stated above. Moreover,
\[
 H_k(L_{b,N}(a))=H_{b,k,N}(a).
\]
We first claim that
\begin{equation}\label{eq:entropymax}
 A_k:=\limsup_{N\to\infty}H_{b,k,N}(a)
          =\max_{\mu\in K}H_k(\mu).
\end{equation}
To see one inequality, choose a subsequence of $N$ along which the
entropy tends to its limsup. By compactness, a further subsequence of
the corresponding empirical measures converges to some $\mu\in K$.
Continuity gives $A_k=H_k(\mu)$. Conversely, every $\mu\in K$ is the
limit of such a subsequence, and hence $H_k(\mu)\leq A_k$. This
proves~\eqref{eq:entropymax}.

All measures in $K$ are invariant, so the entropy inequality above
implies
\[
 A_{k+\ell}=\max_{\mu\in K}H_{k+\ell}(\mu)
       \leq\max_{\mu\in K}H_k(\mu)
          +\max_{\mu\in K}H_\ell(\mu)
       =A_k+A_\ell.
\]
It follows that $A_k/k$ converges and that its limit equals its
infimum. This proves the first equality in~\eqref{eq:entropyidentity}.

It remains to justify the interchange with the maximum. We cannot
simply interchange a limit and a maximum without an argument. For
$j\geq0$, put
\[
 g_j(\mu):=2^{-j}H_{2^j}(\mu),\qquad
 c_j:=\max_{\mu\in K}g_j(\mu).
\]
The inequality $H_{2k}(\mu)\leq2H_k(\mu)$ shows that the continuous
functions $g_j$ decrease pointwise to $h_b$. Consequently, $c_j$
decreases to some limit $c$. Choose a maximizer $\mu_j\in K$ for
each $j$, and then a subsequence $\mu_{j_\ell}$ converging to a
measure $\mu\in K$. For a fixed $s$ and all sufficiently large
$\ell$, we have
\[
 c_{j_\ell}=g_{j_\ell}(\mu_{j_\ell})
                      \leq g_s(\mu_{j_\ell}).
\]
Taking $\ell\to\infty$ gives $c\leq g_s(\mu)$. We may now let
$s\to\infty$ and obtain $c\leq h_b(\mu)$. On the other hand, for
any $\nu\in K$,
\[
 h_b(\nu)\leq g_j(\nu)\leq c_j
 \qquad\text{for every }j,
\]
so $h_b(\nu)\leq c$. Thus $h_b(\mu)=c=\max_Kh_b$. Finally,
$c_j=A_{2^j}/2^j$, whose limit is the limit of the full sequence
$A_k/k$ already obtained by subadditivity. This completes the proof.
\end{proof}

We record Rauzy's characterization in the following form\cite{Rauzy76}.
\begin{theorem}[Rauzy~\cite{Rauzy76}]\label{thm:Rauzy}
For an integer $b\geq2$ and $a\in\R$, the following statements are
equivalent:
\begin{enumerate}
\item $a$ is completely deterministic in base $b$;
\item $x+a\in\calN_b$ for every $x\in\calN_b$;
\item $x\in\calN_b$ if and only if $x+a\in\calN_b$, for every $x\in\R$.
\end{enumerate}
\end{theorem}

Thus Definition~\ref{def:determinism} characterizes precisely the class
$\calD_b$ from~\eqref{eq:Db}. Let us also explain why the one-sided
and two-sided preservation statements agree. Suppose $S\subset\R$
satisfies $S=-S$ and $S+a\subseteq S$. For $x\in S$, we have
$-x\in S$, hence $-x+a\in S$ and therefore $x-a\in S$. This shows
$S-a\subseteq S$ as well. Applying this last inclusion to $x+a$
proves the reverse implication whenever $x+a\in S$. Both $\calN_b$
and $\calN$ satisfy $S=-S$ by Wall's theorem.

In particular, $\calD_b$ and $\calD$ are additive groups. Both groups contain $\Q$ and are stable under multiplication by rational numbers. . From the digit point of view,
a rational number is deterministic because its expansion is eventually
periodic. The empirical limit is supported on a finite periodic orbit,
so its entropy is zero.

\subsection{Proof of Theorem~\ref{thm:addition}: the absolutely normal case}
\label{sec:absolute-addition}

The inclusion $\bigcap_{b\geq2}\calD_b\subseteq\calD$ is immediate.
The converse is more delicate. If $a\notin\calD_b$, there is a
$b$-normal number $u$ for which $u+a$ is not $b$-normal. However, this
alone does not exclude $a$ from $\calD$, since $u$ need not be
absolutely normal. We want to add the same number $z$ to both $u$
and $u+a$. Their difference then remains $a$. The choice of $z$
should improve the normality of the first number without repairing
the failure of normality of the second one.

Our perturbation has independent digits at widely separated positions
in base $b$. There are two separate points to verify. Sparsity gives
a deterministic estimate for the number which is not $b$-normal.
Independence gives the averaging statement needed for the number which
is $b$-normal. Let us first record the latter argument. It is a
finite-future argument of the same kind as the one used in the proof
of Proposition~\ref{prop:irrational}.

\begin{lemma}\label{lem:weightedfuture}
Let $B\geq3$ be an integer, and let $\varepsilon_1,\varepsilon_2,\ldots$
be independent random variables with
$\Prob(\varepsilon_j=0)=\Prob(\varepsilon_j=1)=1/2$. Put
\[
 Z:=\sum_{j=1}^{\infty}\varepsilon_jB^{-j},\qquad
 \nu_B:=\operatorname{law}(Z).
\]
Fix an integer $k$ and a deterministic sequence $(w_n)_{n\geq0}$ with
$|w_n|\leq1$. Then,
\begin{equation}\label{eq:weightedfuture}
 \frac1N\sum_{n=0}^{N-1}w_n
       \bigl(e(kB^nZ)-\widehat\nu_B(k)\bigr)
       \longrightarrow0\qquad\text{almost surely}.
\end{equation}
\end{lemma}

\begin{proof}
The first $n$ digits of $B^nZ$ form an integer. Since $k$ is an
integer as well, we have
\[
 U_n:=e(kB^nZ)=e(kR_n),\qquad
 R_n:=\sum_{j=1}^{\infty}\varepsilon_{n+j}B^{-j}.
\]
The remainder $R_n$ has the same distribution as $Z$, so
$\E U_n=\widehat\nu_B(k)$ for every $n$. The variables $U_n$ are
not independent: their tails overlap. We therefore truncate the tails
before applying a strong law.

For $L\geq1$, let
\[
 R_n^{(L)}:=\sum_{j=1}^L\varepsilon_{n+j}B^{-j},\qquad
 U_n^{(L)}:=e(kR_n^{(L)}),\qquad
 m_L:=\E U_n^{(L)}.
\]
The mean $m_L$ does not depend on $n$. Fix a residue class
$s\in\{0,\ldots,L-1\}$ and write $n_j=s+jL$. The random variable
$U_{n_j}^{(L)}$ depends only on the digits with indices
\[
 s+jL+1,\ldots,s+(j+1)L.
\]
These blocks are disjoint as $j$ varies. Consequently, the variables
\[
 w_{n_j}\bigl(U_{n_j}^{(L)}-m_L\bigr),\qquad j\geq0,
\]
are independent, centered and bounded by $2$ in modulus. They need
not be identically distributed because of the deterministic weights,
but they are bounded martingale differences for the filtration
obtained by revealing one block at a time. The strong law proved in
Section~\ref{sec:multiplication} gives
\[
 \frac1M\sum_{j=0}^{M-1}
        w_{s+jL}\bigl(U_{s+jL}^{(L)}-m_L\bigr)
           \longrightarrow0\qquad\text{almost surely}.
\]
Splitting the average over $0\leq n<N$ into these $L$ residue classes
therefore gives
\begin{equation}\label{eq:truncatedweighted}
 \frac1N\sum_{n=0}^{N-1}w_n\bigl(U_n^{(L)}-m_L\bigr)
            \longrightarrow0\qquad\text{almost surely}.
\end{equation}
Indeed, the contribution of each class is its average multiplied by
the number of its terms divided by $N$, and this last factor is
bounded by $1$.

It remains to control the discarded digits. Uniformly in $n$ and in
the realization of the digits,
\[
 0\leq R_n-R_n^{(L)}
       \leq\sum_{j=L+1}^{\infty}B^{-j}
       =\frac{B^{-L}}{B-1}.
\]
The Lipschitz bound for $e(kt)$ gives
\[
 |U_n-U_n^{(L)}|\leq\frac{2\pi|k|}{B-1}B^{-L},
 \qquad
 |\widehat\nu_B(k)-m_L|
       \leq\frac{2\pi|k|}{B-1}B^{-L}.
\]
Thus centering adds at most a second copy of the truncation error:
\[
 \left|(U_n-\widehat\nu_B(k))-(U_n^{(L)}-m_L)\right|
        \leq\frac{4\pi|k|}{B-1}B^{-L}.
\]
Together with~\eqref{eq:truncatedweighted}, this proves that almost
surely
\[
 \limsup_{N\to\infty}
 \left|\frac1N\sum_{n=0}^{N-1}
           w_n(U_n-\widehat\nu_B(k))\right|
       \leq\frac{4\pi|k|}{B-1}B^{-L}.
\]
We intersect the events of probability one over all positive integers
$L$, and then let $L\to\infty$. This proves~\eqref{eq:weightedfuture}.
The event may depend on the fixed sequence $(w_n)$; no simultaneous
assertion over all such sequences is needed.
\end{proof}

We are now ready for the perturbation lemma which gives the main step
in the absolutely normal case.

\begin{lemma}\label{lem:upgrade}
Fix an integer $b\geq2$. If $u\in\calN_b$ and $v\notin\calN_b$, then
there is a real number $z$ such that
\[
 u+z\in\calN,\qquad v+z\notin\calN_b.
\]
\end{lemma}

\begin{proof}
We first determine how sparse the perturbation has to be. We denote by $A_N^{(b)}(h;x) := \frac1N \sum_{n =0}^{N-1} e(hb^nx)$ the Weyl averages from Weyl's criterion. Since
$v\notin\calN_b$, Weyl's criterion gives an integer $h\neq0$ such that
\[
 \eta:=\limsup_{N\to\infty}|A_N^{(b)}(h;v)|>0.
\]
Thus one fixed Fourier mode already witnesses the failure of
normality. Choose an integer $r\geq2$ sufficiently large that
\begin{equation}\label{eq:sparsechoice}
 \frac{2\pi|h|}{r(b-1)}<\frac\eta2,
\end{equation}
and put $B=b^r$. Consider the compact set
\[
 C_B:=\left\{\sum_{j=1}^{\infty}\varepsilon_jB^{-j}:
                      \varepsilon_j\in\{0,1\}\right\}.
\]
In base $b$, the digits of a point in $C_B$ can be nonzero only at
positions divisible by $r$. We claim that no perturbation from this
set can remove the selected failure of normality of $v$.

To see this, fix any $z=\sum_{j\geq1}\varepsilon_jB^{-j}\in C_B$ and
write $n=qr+s$, where $0\leq s<r$. Then,
\[
 b^nz=\sum_{j=1}^q\varepsilon_jb^{(q-j)r+s}
       +b^s\sum_{j=1}^{\infty}\varepsilon_{q+j}B^{-j}.
\]
The first sum is an integer. The second is nonnegative and at most
$b^s/(B-1)<1$. Hence,
\begin{equation}\label{eq:sparsebound}
 0\leq\{b^nz\}
      =b^s\sum_{j=1}^{\infty}\varepsilon_{q+j}B^{-j}
      \leq\frac{b^s}{B-1}.
\end{equation}
Since $h$ is an integer, it follows that
$
 |e(hb^nz)-1|=|e(h\{b^nz\})-1|
                    \leq2\pi|h|\{b^nz\}.$

Using the definition of the Weyl averages, we obtain
\begin{align*}
 |A_N^{(b)}(h;v+z)-A_N^{(b)}(h;v)|
 &=\left|\frac1N\sum_{n=0}^{N-1}
                e(hb^nv)\bigl(e(hb^nz)-1\bigr)\right|
    \leq\frac{2\pi|h|}{N}\sum_{n=0}^{N-1}\{b^nz\}.
\end{align*}
Each residue class modulo $r$ has asymptotic proportion $1/r$.
Applying~\eqref{eq:sparsebound} in these classes and summing the
geometric progression gives
\begin{align}
 \limsup_{N\to\infty}
 |A_N^{(b)}(h;v+z)-A_N^{(b)}(h;v)|
 &\leq\frac{2\pi|h|}{r(B-1)}\sum_{s=0}^{r-1}b^s=\frac{2\pi|h|}{r(b-1)}<\frac\eta2.
 \label{eq:discrepancystable}
\end{align}
In particular, the triangle inequality, along a subsequence realizing
the limsup which defines $\eta$, yields
\[
 \limsup_{N\to\infty}|A_N^{(b)}(h;v+z)|
       \geq\eta-\frac{2\pi|h|}{r(b-1)}>\frac\eta2.
\]
Another application of Weyl's criterion proves
\begin{equation}\label{eq:allperturbationsbad}
 v+z\notin\calN_b\qquad\text{for every }z\in C_B.
\end{equation}
This part of the argument has no exceptional event: the estimate holds
for every choice of the digits of $z$.

Let us now choose those digits independently with probabilities
$1/2,1/2$, and write
\[
 Z:=\sum_{j=1}^{\infty}\varepsilon_jB^{-j},\qquad
 \nu_B:=\operatorname{law}(Z).
\]
We want to show that $u+Z$ is absolutely normal almost surely. We
begin with the original base. Normality in base $b$ is equivalent to
normality in base $B=b^r$, so $u\in\calN_B$. Fix an integer $k\neq0$
and apply Lemma~\ref{lem:weightedfuture} with
\[
 w_n:=e(kB^nu).
\]
These weights are deterministic, since $u$ is fixed. Multiplying out
the expression in~\eqref{eq:weightedfuture} gives
\begin{equation}\label{eq:goodbasepreserved}
 A_N^{(B)}(k;u+Z)
       -\widehat\nu_B(k)A_N^{(B)}(k;u)
            \longrightarrow0\qquad\text{almost surely}.
\end{equation}
The second term tends to zero by the normality of $u$ in base $B$.
It follows that $A_N^{(B)}(k;u+Z)\to0$ almost surely. We intersect
the events over all nonzero integers $k$ and use Weyl's criterion.
Thus,
\[
 \Prob(u+Z\in\calN_B)=1,
 \qquad\text{and hence}\qquad
 \Prob(u+Z\in\calN_b)=1.
\]
Notice that the Fourier coefficient $\widehat\nu_B(k)$ need not
vanish. The point is that it multiplies an average which tends to
zero. This is where the original normality of $u$ enters the proof.

It remains to obtain normality in the other bases. The law $\nu_B$
is invariant under $T_Bx=Bx\pmod1$, because $T_B$ removes the first
random digit:
\[
 T_BZ=\sum_{j=1}^{\infty}\varepsilon_{j+1}B^{-j}.
\]
It is ergodic for the same reason as the Bernoulli measures in
Section~\ref{sec:multiplication}. More explicitly, the digit coding
identifies this transformation with the one-sided fair Bernoulli
shift; an invariant event is a tail event and hence has probability
zero or one. The coding is one-to-one on the digit sequences used
here, since $B\geq4$ and the digits are only $0$ and $1$. A block
of $m$ such digits has $2^m$ equally likely values, so the entropy is
$\log2>0$.

Let $c\geq2$ be an integer multiplicatively independent of $b$. It is
then also multiplicatively independent of $B=b^r$. We may apply
Theorem~\ref{thm:HS} to $\nu_B$ and the affine diffeomorphism
$t\mapsto u+t$. It follows that
\[
 \Prob(u+Z\in\calN_c)=1.
\]
If $c$ is multiplicatively dependent on $b$, there are positive
integers $p,q$ such that $b^p=c^q$. The power invariance of normality
gives
\[
 \calN_c=\calN_{c^q}=\calN_{b^p}=\calN_b,
\]
so this case has already been covered. Finally, there are only
countably many integer bases. Their probability-one events can be
intersected, and we conclude that
\[
 \Prob(u+Z\in\calN)=1.
\]
Choose a realization $z$ in this event. Every realization belongs to
$C_B$, and hence~\eqref{eq:allperturbationsbad} applies to this same
$z$. We have obtained both required properties.
\end{proof}

We can now finish the additive characterization.

\begin{proof}[Proof of Theorem~\ref{thm:addition}]
The fixed-base statement is Rauzy's theorem, stated above as
Theorem~\ref{thm:Rauzy}. Let $a\in\bigcap_{b\geq2}\calD_b$ and
$x\in\calN$. For each integer base $b$, we have $x\in\calN_b$ and
$a\in\calD_b$, so $x+a\in\calN_b$. Thus $x+a$ belongs to every
$\calN_b$, and hence $a\in\calD$.

For the converse, we argue by contraposition. Suppose
$a\notin\bigcap_{b\geq2}\calD_b$, and choose a base $b$ such that
$a\notin\calD_b$. By the definition of $\calD_b$, there is a real
number $u$ with
\[
 u\in\calN_b,\qquad u+a\notin\calN_b.
\]
Apply Lemma~\ref{lem:upgrade} to $u$ and $v=u+a$. It gives a common
perturbation $z$ for which
\[
 x:=u+z\in\calN,
 \qquad
 x+a=(u+a)+z=v+z\notin\calN_b.
\]
In particular, $x+a\notin\calN$. Thus $a\notin\calD$, which proves
\eqref{eq:absoluteD}. The equivalence of forward preservation and
preservation in both directions follows from the reflection argument
above.
\end{proof}

Combining this theorem with Proposition~\ref{prop:entropy}, we obtain
the explicit digit characterization
\begin{equation}\label{eq:digitcriterion}
 \boxed{\displaystyle
 a\in\calD\quad\Longleftrightarrow\quad
 \text{for every }b\geq2,\quad
 \inf_{k\geq1}\frac1k\limsup_{N\to\infty}H_{b,k,N}(a)=0.}
\end{equation}
The order of the limits matters. Small entropy along just one
subsequence of prefixes is not the condition
in~\eqref{eq:digitcriterion}. Nor does this characterization assert
that $\calD$ consists only of rational numbers; that would be an
additional assertion about the intersection
in~\eqref{eq:absoluteD}.

\section{Normality preserving maps}
\label{sec:C2}

We now turn to Theorems~\ref{thm:2} and~\ref{thm:C11}. The first
part of this section shows why a nonvanishing second derivative is
incompatible with preservation of normality. We then construct a
non-affine preserver whose first derivative is Lipschitz. In the
second argument, the map is affine on many intervals, but the changes
of slope are supported on a compact set of positive measure.

\subsection{Proof of Theorem~\ref{thm:2}: the locally \texorpdfstring{$C^2$}{C2} case}

The main input is the following finite self-similar case of the
theorem of Baker and Banaji \cite[Corollary~1.5]{BB25}. Recall that a
probability measure $\mu$ is self-similar if
\[
 \mu=\sum_{i=1}^mp_i(\varphi_i)_*\mu,
 \qquad p_i>0,\quad\sum_{i=1}^mp_i=1,
\]
for finitely many contracting similarities $\varphi_i$.
Here $(\varphi_i)_*\mu$ denotes the pushforward measure, namely the law
of $\varphi_i(Y)$ when $Y$ has law $\mu$.

\begin{theorem}[Baker--Banaji]\label{thm:BB}
Let $\mu$ be a non-atomic self-similar probability measure supported
in $[0,1]$. Suppose $F$ is $C^2$ on a neighborhood of $[0,1]$ and
$F''(t)\neq0$ throughout $[0,1]$. There are constants $C,\delta>0$
such that
\begin{equation}\label{eq:Fourierdecay}
 \left|\int e(\xi F(t))\,d\mu(t)\right|
       \leq C|\xi|^{-\delta}\qquad(\xi\neq0).
\end{equation}
\end{theorem}

The quantity on the left is the Fourier transform of $F_*\mu$.
We shall use that the stated decay forces almost every point for this
pushforward measure to be absolutely normal. Let us include the
argument, since it makes clear how the Fourier estimate enters the
proof.

\begin{lemma}\label{lem:decaynormal}
If a probability measure $\nu$ on $\R$ satisfies
$|\widehat\nu(\xi)|\leq C|\xi|^{-\delta}$ for some $C,\delta>0$
and all $\xi\neq0$, then $\nu$-almost every point is absolutely normal.
\end{lemma}

\begin{proof}
Fix an integer base $b\geq2$ and a nonzero integer $h$. By Weyl's
criterion, we want to prove that $A_N^{(b)}(h;x)\to0$ for
$\nu$-almost every $x$. We begin with its second moment. Expanding
the square gives
\[
 \int|A_N^{(b)}(h;x)|^2\,d\nu(x)
   =\frac1{N^2}\sum_{m,n=0}^{N-1}
                    \widehat\nu\bigl(h(b^n-b^m)\bigr).
\]
The $N$ diagonal terms equal $1$. The off-diagonal terms occur in
complex conjugate pairs. Taking absolute values and applying the
Fourier estimate, we obtain
\begin{equation}\label{eq:normalitysecondmoment}
 \int|A_N^{(b)}(h;x)|^2\,d\nu(x)
 \leq\frac1N+\frac{2C|h|^{-\delta}}{N^2}
       \sum_{0\leq m<n<N}(b^n-b^m)^{-\delta}.
\end{equation}
For $m<n$, we have $b^m\leq b^{n-1}$, and hence
\[
 b^n-b^m\geq(1-b^{-1})b^n.
\]
There are $n$ possible indices $m<n$ for each fixed $n$. Therefore,
\[
 \sum_{0\leq m<n<N}(b^n-b^m)^{-\delta}
 \leq(1-b^{-1})^{-\delta}
                     \sum_{n=1}^{\infty}n b^{-\delta n}<\infty.
\]
The bound is independent of $N$. In particular, there is a constant
$C_0$, depending on $b,h$ and $\nu$, such that
\[
 \int|A_N^{(b)}(h;x)|^2\,d\nu(x)\leq\frac{C_0}{N}.
\]

As in the martingale argument in Section~\ref{sec:multiplication},
we first pass to square indices. For every $\varepsilon>0$,
\[
 \sum_{j=1}^{\infty}
 \nu\{x:\ |A_{j^2}^{(b)}(h;x)|>\varepsilon\}
 \leq\frac{C_0}{\varepsilon^2}
                      \sum_{j=1}^{\infty}\frac1{j^2}<\infty.
\]
Chebyshev's inequality and Borel--Cantelli, followed by the countable
intersection over $\varepsilon=1,1/2,1/3,\ldots$, imply
$A_{j^2}^{(b)}(h;x)\to0$ for $\nu$-almost every $x$.

To pass to all indices, suppose $j^2\leq N<(j+1)^2$. Every summand
in a Weyl average has modulus $1$, so separating the first $j^2$
terms from the remaining terms gives, for every $x$,
\[
 |A_N^{(b)}(h;x)-A_{j^2}^{(b)}(h;x)|
       \leq2\frac{N-j^2}{N}
       \leq2\frac{2j+1}{j^2}\longrightarrow0.
\]
Thus $A_N^{(b)}(h;x)\to0$ almost surely. Finally, there are only
countably many pairs $(b,h)$ under consideration. Intersecting their
full-measure sets and applying Weyl's criterion proves absolute
normality.
\end{proof}

We are now ready to prove the characterization. The important point
is to apply Theorem~\ref{thm:BB} to a local inverse of $f$, not to
$f$ itself. We want a normal input whose image is not normal, so we
start with a measure on nonnormal points in the range of $f$ and
pull it back.

\begin{proof}[Proof of Theorem~\ref{thm:2}]
Suppose first that $f\in C^2_{\mathrm{loc}}(I)$ is not affine-linear.
Then $f''$ is not identically zero. Indeed, if $f''\equiv0$ on the
interval $I$, then $f'$ is constant and $f$ is affine-linear.
Choose a point where $f''$ does not vanish. By continuity, there is
an open interval $U\subset I$ on which $f''$ has a fixed nonzero
sign. The derivative $f'$ is strictly monotone on $U$, so it can
vanish at most once there. After restricting to a smaller nonempty
open interval $J\subset U$, we may assume that both $f'$ and $f''$
are nonzero throughout $J$.

The function $f$ is strictly monotone on $J$. Its restriction is a
$C^2$ diffeomorphism from $J$ onto the open interval $V=f(J)$.
Let $g:V\to J$ be its inverse. Differentiating $f(g(y))=y$ once
and twice gives
\[
 f'(g(y))g'(y)=1,
 \qquad
 f''(g(y))(g'(y))^2+f'(g(y))g''(y)=0.
\]
Consequently,
\begin{equation}\label{eq:inversesecond}
 g'(y)=\frac1{f'(g(y))},\qquad
 g''(y)=-\frac{f''(g(y))}{(f'(g(y)))^3}\neq0
 \qquad(y\in V).
\end{equation}
Choose rational numbers $r$ and $s>0$ such that
$[r,r+s]\subset V$. This is possible because $V$ is a nonempty
open interval. Define
\[
 F(t):=g(r+st),\qquad 0\leq t\leq1.
\]
Since the compact interval $[r,r+s]$ lies inside $V$, the function
$F$ is defined and $C^2$ on a neighborhood of $[0,1]$. Moreover,
$F''(t)=s^2g''(r+st)\neq0$ there after possibly reducing that
neighborhood.

Let us first consider preservation of $\calN_b$. Take the random
number $Y$ from~\eqref{eq:Y}, with a nonuniform probability vector
as in~\eqref{eq:p}. Its law $\mu_{b,\mathbf p}$ is self-similar
for the maps
\[
 \varphi_d(t)=\frac{t+d}{b},\qquad d=0,\ldots,b-1,
\]
with weights $p_d$. This follows by separating the first digit of
$Y$ from the remaining digits. It is also non-atomic. To see this,
put $p_*:=\max_dp_d<1$. Each prescribed initial block of $n$ digits
has probability at most $p_*^n$. A point has at most two base-$b$
expansions, so its mass is at most $2p_*^n$ for every $n$, and
hence is zero.

All assumptions of Theorem~\ref{thm:BB} are therefore satisfied by
$\mu_{b,\mathbf p}$ and $F$. The law of $F(Y)$ has polynomial
Fourier decay. By Lemma~\ref{lem:decaynormal},
\[
 \Prob(F(Y)\in\calN)=1.
\]
On the other hand, the digits of $Y$ have limiting frequencies
$p_d$, which are not all $1/b$. Thus $Y\notin\calN_b$ almost
surely. Since $r,s$ are rational and $s\neq0$, Wall's theorem gives
\[
 f(F(Y))=r+sY\notin\calN_b\qquad\text{almost surely}.
\]
Both assertions hold on the intersection of two probability-one
events. Also $F(Y)\in J\subset I$ for every realization. Choosing
one realization in that intersection gives a $b$-normal input in
$I$ whose image is not $b$-normal, a contradiction.

For preservation of $\calN$, use the binary random variable
from~\eqref{eq:Y2} in the same construction. Again $F(Y)$ is
absolutely normal almost surely, whereas $r+sY$ is almost surely
not normal in base $2$, and hence is not absolutely normal. This
gives the same contradiction. We conclude that in either case $f$
must be affine-linear on $I$.

It remains to determine the admissible coefficients. Write
$f(x)=ax+c$. If $a=0$, the relevant normal class meets $I$, since
it has full Lebesgue measure. Preservation therefore holds exactly
when the constant $c$ belongs to that class. From now on, let
$a\neq0$.

We first exclude irrational $a$. Choose rationals $r$ and $s>0$
with $[r,r+s]\subset f(I)$, and consider
\begin{equation}\label{eq:affineinversewitness}
 X:=\frac{r+sY-c}{a}
    =\frac{s}{a}Y+\frac{r-c}{a}.
\end{equation}
The choice of the range interval guarantees that $X\in I$ for
every realization of $Y$. If $a\notin\Q$, then $s/a\notin\Q$,
since otherwise $a=s/(s/a)$ would be rational. For the fixed-base
case, take $Y$ from~\eqref{eq:Y}. Proposition~\ref{prop:irrational},
applied with $\gamma=s/a$ and $\eta=(r-c)/a$, gives
\[
 \Prob(X\in\calN_b)=1.
\]
But $f(X)=r+sY\notin\calN_b$ almost surely, again by Wall's
theorem. This contradicts preservation on $I$. For the absolutely
normal case, take $Y$ from~\eqref{eq:Y2} and use
Proposition~\ref{prop:irrational2} instead. Then $X$ is absolutely
normal almost surely and $f(X)$ is not normal in base $2$ almost
surely. Thus both cases require $a\in\Q\setminus\{0\}$.

Finally, we identify the translation parameter $c$. The assumption
only concerns inputs in $I$, so let us explain why it gives a global
translation preserver. Write $S$ for either $\calN_b$ or $\calN$,
and let $y\in S$ be arbitrary. By Wall's theorem, $y/a\in S$.
The interval $I-y/a$ is nonempty and open, so it contains a rational
number $q$. Then $y/a+q\in I\cap S$, and the preservation assumption
gives
\[
 f(y/a+q)=y+aq+c\in S.
\]
Since $aq$ is rational, subtracting it does not change normality.
Hence $y+c\in S$. This holds for every $y\in S$, and therefore
$c\in\calD_b$ in the fixed-base case and $c\in\calD$ in the
absolutely normal case.

Conversely, a normal constant plainly preserves the relevant class.
If $a\in\Q\setminus\{0\}$ and $c$ belongs to the corresponding
class of additive preservers, then $x\mapsto ax$ preserves normality
by Wall's theorem and $x\mapsto x+c$ preserves it by definition.
Their composition $x\mapsto ax+c$ has the same property, also after
restriction to $I$. This completes the proof of both characterizations.
\end{proof}

\subsection{Proof of Theorem~\ref{thm:C11}: a non-affine \texorpdfstring{$C^{1,1}$}{C1,1} example}
\label{sec:C11}

We now prove Theorem~\ref{thm:C11}. In fact, we first construct a map
preserving nonnormality, and then take its inverse as counterexample.
Let us explain the reason for this choice. Suppose $K$ is a compact
set consisting entirely of absolutely normal numbers, and suppose
$g$ is rational-affine, with nonzero slope, on every component of
$\R\setminus K$. If $y$ is not normal in some base $b$, then
$y\notin K$, so Wall's theorem applies to $g$ on the complementary
interval containing $y$. It follows that $g(y)$ is still not normal
in base $b$. The inverse map will therefore preserve normality.

The difficulty is to make $g$ non-affine while keeping its first
derivative Lipschitz and all of its affine coefficients on the
complementary intervals rational. We shall achieve this by prescribing
two moments of a bounded density. The Lebesgue measure is denoted by
$\lambda$ in the following.

\begin{lemma}\label{lem:rationaljets}
Let $K\subset(0,1)$ be a compact set of positive measure with empty
interior, and suppose
\begin{equation}\label{eq:fullsupportK}
 K=\supp(\lambda|_K).
\end{equation}
There is a non-affine increasing $C^{1,1}$ diffeomorphism
$g:\R\to\R$ such that on every component $J$ of $\R\setminus K$,
\begin{equation}\label{eq:rationalaffinegaps}
 g(x)=a_Jx+c_J\quad(x\in J),\qquad
 a_J\in\Q\cap[1,5/2],\quad c_J\in\Q.
\end{equation}
Moreover, $1\leq g'\leq5/2$ and $\Lip(g')\leq3/2$.
\end{lemma}

\begin{proof}
We look for $g$ in the form
\begin{equation}\label{eq:gconstruction}
 g(x):=x+\int_K(x-t)_+w(t)\,dt,
 \qquad \frac12\leq w(t)\leq\frac32
             \quad\text{for almost every }t\in K,
\end{equation}
where $(x-t)_+:=\max\{x-t,0\}$. For the moment, suppose such a
bounded density $w$ has been chosen. Differentiation under the
integral gives
\begin{equation}\label{eq:gderivative}
 g'(x)=1+\int_{K\cap(-\infty,x)}w(t)\,dt,
 \qquad
 g''(x)=w(x)\1_K(x)\quad\text{almost everywhere}.
\end{equation}
The first formula holds at every $x$: the only possible failure of
differentiability of $(x-t)_+$ occurs at the single point $t=x$,
which has Lebesgue measure zero. The bounded density allows dominated
convergence. The second formula follows from the fundamental theorem
for Lebesgue integrals.

In particular, the first derivative is continuous. More precisely,
if $x<y$, then
\[
 0\leq g'(y)-g'(x)=\int_{K\cap[x,y)}w(t)\,dt
                  \leq\frac32(y-x).
\]
Also, since $K\subset(0,1)$,
\[
 1\leq g'(x)\leq1+\frac32\lambda(K)\leq\frac52.
\]
Thus $g\in C^{1,1}$ with the desired derivative bounds. Since
$g'\geq1$, the map is strictly increasing and tends to $-\infty$
and $+\infty$ at the respective ends of the real line. It is a
bijection of $\R$ onto itself, and its inverse is continuously
differentiable. Moreover, $g$ is not affine: its derivative equals
$1$ to the left of $K$ and equals $1+\int_Kw>1$ to the right.

We now examine its restriction to the complementary intervals.
Let $J=(\alpha,\beta)$ be a bounded component of $\R\setminus K$.
For $x\in J$, every $t\in K$ lies either to the left of $x$ or to
the right of $x$, and the part lying to the left is independent of
$x\in J$. Consequently,~\eqref{eq:gconstruction} becomes
\begin{equation}\label{eq:gaponmoments}
 g(x)=(1+M_J)x-P_J\qquad(x\in J),
\end{equation}
where
\[
 M_J:=\int_{K\cap(-\infty,\alpha)}w(t)\,dt,
 \qquad
 P_J:=\int_{K\cap(-\infty,\alpha)}t w(t)\,dt.
\]
Single endpoints do not affect these integrals. Thus both affine
coefficients are rational as soon as the mass $M_J$ and the first
moment $P_J$ are rational. On the left unbounded component we already
have $g(x)=x$. On the right unbounded component the same formula holds
with the total moments
\[
 M:=\int_Kw(t)\,dt,
 \qquad P:=\int_Kt w(t)\,dt.
\]
Our task is therefore to choose $w$ so that all these pairs of
moments are rational, while retaining the pointwise bounds
in~\eqref{eq:gconstruction}.

Let us first record the elementary freedom available for correcting
two moments. Suppose $E_1,E_2\subset K$ have positive measure and
are separated, with $E_1$ to the left of $E_2$. Write
\[
 m_i:=\lambda(E_i),\qquad
 \tau_i:=\frac1{m_i}\int_{E_i}t\,dt\qquad(i=1,2).
\]
The numbers $\tau_i$ are the respective mean positions, and the
separation gives $\tau_1<\tau_2$. A correction
$v=u_1\1_{E_1}+u_2\1_{E_2}$ changes mass and first moment by
\begin{equation}\label{eq:momentmatrix}
 \begin{pmatrix}\int_Kv(t)\,dt\\[2pt]\int_Kt v(t)\,dt\end{pmatrix}
 =\begin{pmatrix}m_1&m_2\\m_1\tau_1&m_2\tau_2\end{pmatrix}
        \begin{pmatrix}u_1\\u_2\end{pmatrix}.
\end{equation}
The determinant is $m_1m_2(\tau_2-\tau_1)>0$. Hence, for every
prescribed pair of changes $(\Delta M,\Delta P)$, there is a unique
choice of $u_1,u_2$ which realizes it. For these fixed sets there is
a finite constant $C_E$ such that
\[
 \|v\|_{L^\infty(K)}=\max\{|u_1|,|u_2|\}
       \leq C_E\max\{|\Delta M|,|\Delta P|\}.
\]
In particular, the correction can be made arbitrarily small by asking
for sufficiently small changes in the two moments.

We shall repeatedly choose two such separated sets inside a given
positive-measure part of $K$. This is always possible: a measure
obtained by restricting Lebesgue measure has no atoms. A
positive-measure part therefore has two distinct points in its
support, and small disjoint neighborhoods of those points give the
required sets of positive measure.

Start with the constant density $w\equiv1$. Choose two separated
positive-measure subsets of $K$. By the density of $\Q^2$ in $\R^2$,
there is a rational pair arbitrarily close to the current total mass
and first moment. Formula~\eqref{eq:momentmatrix} allows us to change
the total moments to that rational pair with a correction of
$L^\infty$ norm less than $1/8$. Denote the resulting density by
$w_0$. Thus,
\[
 \|w_0-1\|_{L^\infty(K)}<\frac18,
 \qquad
 \int_Kw_0\in\Q,\qquad \int_Kt w_0(t)\,dt\in\Q.
\]
We shall keep these total moments fixed from now on.

The bounded components of $\R\setminus K$ form a countable family;
each of them contains a distinct rational number. Enumerate them as
$J_1,J_2,\ldots$, and choose a point $q_j\in J_j$ for every $j$.
Suppose $w_{n-1}$ has been constructed so that the total moments and
the two moments to the left of $J_1,\ldots,J_{n-1}$ are rational.
We explain how to obtain the same property at $J_n$ without changing
any of those earlier values.

Write $J_n=(\alpha_n,\beta_n)$. Among the previously chosen points
$q_1,\ldots,q_{n-1}$, let $\ell$ be the nearest one to the left of
$q_n$, taking $\ell=0$ if there is none. Let $r$ be the nearest one
to the right, taking $r=1$ if there is none. Thus,
\[
 \ell<\alpha_n<\beta_n<r,
 \qquad
 \{q_1,\ldots,q_{n-1}\}\cap(\ell,r)=\varnothing.
\]
We will make one correction in $K\cap(\ell,\alpha_n)$ and a
compensating correction in $K\cap(\beta_n,r)$.

Both of these sets have positive measure. Let us check this on the
left. The endpoint $\alpha_n$ belongs to $K$. Choose a sufficiently
small neighborhood of $\alpha_n$ which lies in $(\ell,\beta_n)$.
By~\eqref{eq:fullsupportK}, this neighborhood has positive $K$-measure.
Its part to the right of $\alpha_n$ lies in the gap $J_n$, and the
single point $\alpha_n$ has zero measure. Thus the positive
$K$-measure must lie to its left, in $(\ell,\alpha_n)$. The argument
at $\beta_n$ is the same. We may therefore choose two separated
positive-measure subsets on each side:
\[
 E_1^-,E_2^-\subset K\cap(\ell,\alpha_n),
 \qquad
 E_1^+,E_2^+\subset K\cap(\beta_n,r).
\]

Let $(M_n,P_n)$ denote the moments of $w_{n-1}$ to the left of
$J_n$. Choose a rational pair $(M_n',P_n')$ close to $(M_n,P_n)$,
and put
\[
 \Delta M:=M_n'-M_n,
 \qquad \Delta P:=P_n'-P_n.
\]
Using $E_1^-,E_2^-$ in~\eqref{eq:momentmatrix}, choose a correction
$v_n^-$ with mass $\Delta M$ and first moment $\Delta P$. Using
$E_1^+,E_2^+$, choose a correction $v_n^+$ with mass $-\Delta M$
and first moment $-\Delta P$. Their sum
\[
 v_n:=v_n^-+v_n^+
\]
therefore satisfies
\begin{equation}\label{eq:zerototalcorrection}
 \int_Kv_n(t)\,dt=0,
 \qquad
 \int_Kt v_n(t)\,dt=0.
\end{equation}
The correction on the right does not contribute to the moments to
the left of $J_n$. Hence those two moments change by exactly
$(\Delta M,\Delta P)$ and become $(M_n',P_n')$, as desired.

Let us verify explicitly that every earlier prescription remains
unchanged. For $j<n$, the moments to the left of $J_j$ can be
computed by integrating over $K\cap(-\infty,q_j)$, because
$q_j\in J_j$ and the gap contains no part of $K$. If $q_j\leq\ell$,
neither correction contributes to these integrals. If $q_j\geq r$,
both corrections contribute, and their contributions cancel by
\eqref{eq:zerototalcorrection}. These are the only possibilities,
since there is no earlier $q_j$ inside $(\ell,r)$. The total moments
are unchanged for the same reason.

The two inverse matrices used for the left and right corrections are
fixed at this stage. By taking the rational pair $(M_n',P_n')$
sufficiently close to $(M_n,P_n)$, we can therefore ensure
\begin{equation}\label{eq:correctionbound}
 \|v_n\|_{L^\infty(K)}<2^{-n-3}.
\end{equation}
There is no need for a bound on the inverse matrices uniform in $n$:
at each stage, the density of $\Q^2$ allows the changes of moments
to be chosen as small as necessary. Set $w_n=w_{n-1}+v_n$. This
completes the induction.

The estimates~\eqref{eq:correctionbound} show that $w_n$ converges
in $L^\infty(K)$ to a bounded density $w$. Moreover,
\[
 \|w-1\|_{L^\infty(K)}
 \leq\|w_0-1\|_{L^\infty(K)}
                 +\sum_{n=1}^{\infty}\|v_n\|_{L^\infty(K)}
 <\frac18+\sum_{n=1}^{\infty}2^{-n-3}
 =\frac14.
\]
In particular, the bounds required
in~\eqref{eq:gconstruction} hold. For each fixed $J_j$, its two
moments are assigned rational values at stage $j$ and remain exactly
unchanged at every subsequent stage. Integration of $w_n$ and of
$tw_n$ passes to the limit by $L^\infty$ convergence on a finite
measure space. Thus the limiting moments at $J_j$ equal that same
rational pair. The total moments likewise remain equal to their
rational values assigned at stage $0$. We are not merely taking a
limit of varying rational numbers; each prescribed value is fixed
from a finite stage onward.

Define $g$ by~\eqref{eq:gconstruction} with this density $w$.
Formula~\eqref{eq:gaponmoments} gives rational affine coefficients
on every bounded complementary interval. The total moments give the
same conclusion on the right unbounded interval, and $g(x)=x$ on
the left one. The derivative bounds and the fact that $g$ is a
non-affine increasing diffeomorphism have already been verified.
This proves the lemma.
\end{proof}

It remains to choose the set $K$ and check the properties of the
inverse map.

\begin{proof}[Proof of Theorem~\ref{thm:C11}]
The set $\calN$ is Borel: by Weyl's criterion it is defined by
countably many convergence conditions on continuous functions.
It has full Lebesgue measure by Borel's theorem. Inner regularity
therefore gives a compact set
\[
 K_0\subset\calN\cap(0,1),\qquad\lambda(K_0)>0.
\]
Put $K:=\supp(\lambda|_{K_0})$. This is a compact subset of $K_0$.
Passing from $K_0$ to this support removes only a null set: the
complement of the support can be covered by countably many open
intervals having zero $\lambda|_{K_0}$-measure. Consequently,
$\lambda|_K=\lambda|_{K_0}$, so $K$ has positive measure and
satisfies~\eqref{eq:fullsupportK}. It still consists entirely of
absolutely normal numbers. Finally, it has empty interior, because
every nonempty open interval contains a rational number and no
rational number is normal.

Apply Lemma~\ref{lem:rationaljets} to this $K$, and let $g$ be the
resulting map. Define $f=g^{-1}$. We first check its regularity.
Since $g$ is continuously differentiable and $g'\geq1$, the inverse
function theorem gives
\[
 f'(x)=\frac1{g'(f(x))},\qquad \frac25\leq f'(x)\leq1.
\]
In particular, $f$ is globally Lipschitz with constant at most $1$.
For any $x,y\in\R$, we obtain
\begin{align*}
 |f'(x)-f'(y)|
 &=\frac{|g'(f(x))-g'(f(y))|}
             {g'(f(x))g'(f(y))}
 \leq |g'(f(x))-g'(f(y))|\\
 &\leq\frac32|f(x)-f(y)|
 \leq\frac32|x-y|.
\end{align*}
Thus $f\in C^{1,1}$, and both $f$ and its inverse $g$ have globally
Lipschitz derivatives. The map $f$ is increasing and is not affine,
since an affine inverse would force $g$ to be affine as well.

Fix now any integer base $b\geq2$. Suppose $y\notin\calN_b$.
Since $K\subset\calN\subset\calN_b$, the point $y$ lies outside
$K$. It belongs to a component $J$ of $\R\setminus K$, and on that
interval
\[
 g(y)=a_Jy+c_J,
 \qquad a_J\in\Q\setminus\{0\},\quad c_J\in\Q.
\]
Wall's theorem gives the equivalence
$g(y)\in\calN_b$ if and only if $y\in\calN_b$. In particular,
\[
 y\notin\calN_b\quad\Longrightarrow\quad g(y)\notin\calN_b.
\]
Let now $x\in\calN_b$ and set $y=f(x)$. If $y$ were not
$b$-normal, the preceding implication would give
$x=g(y)\notin\calN_b$, a contradiction. Therefore
$f(x)\in\calN_b$, proving~\eqref{eq:C11preserve}.

The same construction and the same set $K$ work for every integer
base $b$. Hence if $x$ is absolutely normal, its image $f(x)$ is
normal in every base and is absolutely normal as well. This proves
all assertions of Theorem~\ref{thm:C11}.
\end{proof}
\paragraph{Acknowledgments.}
The author thanks Christoph Aistleitner for suggesting this problem and making the author aware of the reference \cite{DGW24}.
This work was funded by the Deutsche Forschungsgemeinschaft
(DFG, German Research Foundation) -- 558731723.

\paragraph{AI Declaration.}
AI assistance has been used in the proof of Theorem~\ref{thm:2}: an LLM identified that a former proof for analytic maps can be generalized to $C^2$-functions. Parts of the proof of Theorem~\ref{thm:addition} and Theorem~\ref{thm:C11} have been typeset via AI based on handwritten notes without changing the arguments by the author. Further uses are limited to language and typographical corrections, stylistic improvements and minor improvements.


\begin{thebibliography}{99}

\bibitem{ABSS17}
C. Aistleitner, V. Becher, A.-M. Scheerer and T. A. Slaman.
\emph{On the construction of absolutely normal numbers.}
Acta Arith. \textbf{180} (2017), 333--346.

\bibitem{ABS22}
A. Algom, S. Baker and P. Shmerkin.
\emph{On normal numbers and self-similar measures.}
Adv. Math. \textbf{399} (2022), 108276.

\bibitem{BBCP04}
D. H. Bailey, J. M. Borwein, R. E. Crandall and C. Pomerance.
\emph{On the binary expansions of algebraic numbers.}
J. Th\'eor. Nombres Bordeaux \textbf{16} (2004), 487--518.

\bibitem{BB25}
S. Baker and A. Banaji.
\emph{Polynomial Fourier decay for fractal measures and their pushforwards.}
Math. Ann. \textbf{392} (2025), 209--261.

\bibitem{BF02}
V. Becher and S. Figueira.
\emph{An example of a computable absolutely normal number.}
Theoret. Comput. Sci. \textbf{270} (2002), 947--958.

\bibitem{BHS13}
V. Becher, P. A. Heiber and T. A. Slaman.
\emph{A polynomial-time algorithm for computing absolutely normal numbers.}
Inform. and Comput. \textbf{232} (2013), 1--9.

\bibitem{BD25}
V. Bergelson and T. Downarowicz.
\emph{On preservation of normality and determinism under arithmetic operations.}
Preprint, arXiv:2506.12929, 2025.

\bibitem{Borel09}
\'E. Borel.
\emph{Les probabilit\'es d\'enombrables et leurs applications arithm\'etiques.}
Rend. Circ. Mat. Palermo \textbf{27} (1909), 247--271.

\bibitem{Bug12}
Y. Bugeaud.
\emph{Distribution modulo one and Diophantine approximation.}
Cambridge Tracts in Mathematics 193, Cambridge University Press, 2012.

\bibitem{Cass59}
J. W. S. Cassels.
\emph{On a problem of Steinhaus about normal numbers.}
Colloq. Math. \textbf{7} (1959), 95--101.

\bibitem{Champernowne1933}
D. G. Champernowne.
\emph{The construction of decimals normal in the scale of ten.}
J. London Math. Soc. \textbf{8} (1933), 254--260.

\bibitem{Chow67}
Y. S. Chow.
\emph{On a strong law of large numbers for martingales.}
Ann. Math. Statist. \textbf{38} (1967), 610.

\bibitem{CE46}
A. H. Copeland and P. Erd\H{o}s.
\emph{Note on normal numbers.}
Bull. Amer. Math. Soc. \textbf{52} (1946), 857--860.

\bibitem{DGW24}
Y. Dayan, A. Ganguly and B. Weiss.
\emph{Random walks on tori and normal numbers in self-similar sets.}
Amer. J. Math. \textbf{146} (2024), 467--493.
Preprint version: arXiv:2002.00455v3.

\bibitem{DT97}
M. Drmota and R. F. Tichy.
\emph{Sequences, discrepancies and applications.}
Lecture Notes in Mathematics 1651, Springer, 1997.

\bibitem{Harman98}
G. Harman.
\emph{Metric number theory.}
London Mathematical Society Monographs, New Series 18,
Oxford University Press, 1998.

\bibitem{HS15}
M. Hochman and P. Shmerkin.
\emph{Equidistribution from fractal measures.}
Invent. Math. \textbf{202} (2015), 427--479.

\bibitem{ManaiAdv}
C. Manai.
\emph{Super-dense sets and their role in the theory of normal numbers}
Adv. Math. 503, 111205 (2026).

\bibitem{ManaiTranscendence}
C. Manai.
\emph{Transcendence meets normality: Construction of transcendentally normal numbers.}
To appear in Proc. London Math. Soc. (2026+) Preprint, arXiv:2508.09319.

\bibitem{ManaiBernoulli}
C. Manai.
\emph{Digit mixing under polynomial maps.}
Preprint, arXiv:2606.08325, 2026.

\bibitem{Rauzy76}
G. Rauzy.
\emph{Nombres normaux et processus d\'eterministes.}
Acta Arith. \textbf{29} (1976), 211--225.

\bibitem{Schmidt1960}
W. M. Schmidt.
\emph{On normal numbers.}
Pacific J. Math. \textbf{10} (1960), 661--672.

\bibitem{Schmidt1962}
W. M. Schmidt.
\emph{\"Uber die Normalit\"at von Zahlen zu verschiedenen Basen.}
Acta Arith. \textbf{7} (1962), 299--309.

\bibitem{Wall49}
D. D. Wall.
\emph{Normal numbers.}
Ph.D. thesis, University of California, Berkeley, 1949.
\vspace{-0.05cm}
\bibitem{Weyl1916}
H. Weyl.
\emph{\"Uber die Gleichverteilung von Zahlen mod. Eins.}
Math. Ann. \textbf{77} (1916), 313--352.

\end{thebibliography}
\end{document}